\documentclass[a4paper, 12pt, draft]{article}

\usepackage[russian]{babel}
\usepackage[cp1251]{inputenc}
\usepackage{amsmath,latexsym}
\usepackage[psamsfonts]{amssymb}
\usepackage{amssymb,amsmath,amsthm,amsfonts}
\usepackage{indentfirst}

\newtheorem{theorem}{Теорема}
\newtheorem{lemma}{Лемма}
\newtheorem{corolarry}{Следствие}

\begin{document}

\begin{center}
{\bf А.\,К. Уринов, Ш.\,Т.~Каримов}\\
{\bf РЕШЕНИЕ АНАЛОГА ЗАДАЧИ КОШИ ДЛЯ ИТЕРИРОВАННОГО МНОГОМЕРНОГО УРАВНЕНИЯ КЛЕЙНА-ГОРДОНА-ФОКА С ОПЕРАТОРОМ БЕССЕЛЯ}
\end{center}



\begin{abstract}
Исследуется аналог задачи Коши для итерированного многомерного уравнения Клейна-Гордона-Фока с оператором Бесселя, действующим по времени. Применяя обобщенный оператор Эрдейи-Кобера дробного порядка, поставленная задача сведена к задаче Коши для поливолнового уравнения. Построена явная формула решения этой задачи методом сферического среднего. На основе полученного решения найдено интегральное представление решения поставленной задачи.

{\bf Ключевые слова:} Задача Коши, уравнения Клейна-Гордона-Фока, обобщенный оператор Эрдейи-Кобера,  оператор Бесселя.
\end{abstract}

\section{\large Введение. Постановка задачи}

Предметом исследования является решение аналога задачи Коши для итерированного многомерного уравнения Клейна-Гордона-Фока с оператором Бесселя

\begin{equation}\label{eq1}
L_{\gamma ,\lambda }^m (u) \equiv \left( {B_\gamma ^t  - \Delta  + \lambda ^2 } \right)^m u = 0,\,\,(x,t) \in \Omega,
\end{equation}
удовлетворяющего начальным условиям
\begin{equation}\label{ic2}
\left. {\dfrac{{\partial ^{2k} u}}{{\partial t^{2k} }}} \right|_{t = 0}  = \varphi _k (x),\,\,\left. {\dfrac{{\partial ^{2k + 1} u}}{{\partial t^{2k + 1} }}} \right|_{t = 0}  = 0,\,\,x \in R^n, \,\,k = \overline {0,m - 1},
\end{equation}
\begin{equation}\label{ic3}
\left. {\dfrac{{\partial ^{2k} u}}{{\partial t^{2k} }}} \right|_{t = 0}  = 0,\,\,\left. {t^{2\gamma  + 1} \dfrac{{\partial ^{2k + 1} u}}{{\partial t^{2k + 1} }}} \right|_{t = 0}  = \psi _k (x),\,\,x \in R^n ,\,\,k = \overline {0,m - 1},
\end{equation}
где $u = u(x,t),$ $\Omega=\{(x,t): x \in R^n, t \in R, t>0 \},$ $m \in N,$ $L_{\gamma ,\lambda }^m  = L_{\gamma ,\lambda }^{} \left( {L_{\gamma ,\lambda }^{m - 1} } \right),$  $B_\gamma ^t  \equiv \partial ^2 /\partial t^2  + [(2\gamma  + 1)/t](\partial /\partial t),$  -оператор Бесселя, $\Delta   \equiv \sum\limits_{k = 0}^n ({\partial ^2}/{\partial x_k^2 }) $  - оператор Лапласа,  $\gamma ,\lambda  \in R,$  $\gamma  >  -1/2,$ $\varphi _k (x)$  и  $\psi _k (x),$  $k = \overline {0,m - 1} $  - заданные дифференцируемые функции.

Исследование итерированных уравнений высокого порядка представляет собой естественный дальнейший этап на пути теоретических обобщений. Ценность получаемых при этом теоретических результатов существенно возрастает в связи с тем, что подобные уравнения или их частные случаи встречаются в приложениях. Уравнение~\eqref{eq1} при $m = 1$  и $\gamma  =  - 1/2,$ переходит в уравнение Клейна-Гордона-Фока, которое является релятивистской версией уравнения Шрёдингера и представляет собой дифференциальное уравнение в частных производных, относящееся к классу гиперболических уравнений второго порядка. Оно описывает динамику релятивистской квантовой системы~\cite{BLOT}. Кроме того, оно является обобщением волнового уравнения, подходящего для описания без массовых скалярных и векторных полей.

В случае $\lambda  = 0,$ в работах \cite{Iva1}, \cite{Ald1} получено представление решения поставленной задачи. При этом решение конструировано, используя решение задачи Коши для уравнения Эйлера-Пуассона-Дарбу (ЭПД). Последнее основано на том факте, что среднее сферическое достаточно гладкой функции удовлетворяет уравнению Дарбу~\cite{Wei1}.   Отправляясь от этого факта \cite{Kur1}, получают формулу Кирхгофа для решения уравнения ЭПД, в том числе, для волнового уравнения.
Однако, когда $\lambda  \ne 0,$ этот факт не имеет места, поэтому этот подход неприменим.

В данной работе, в отличие от цитируемых источников, для решения поставленной задачи применим другой подход. А именно, учитывая специфику уравнений с сингулярными коэффициентами, используем  обобщенный оператор интегрирования дробного порядка Эрдейи-Кобера. Данный подход для решения поставленной задачи при $m = 1$  применен в работе \cite{UK}. Применение обобщенного оператора интегрирования дробного порядка Эрдейи-Кобера позволяет сводить уравнения с младшим членом $\lambda ^2 u$  и с сингулярным оператором Бесселя, который действует по одной или нескольким переменным,  к не сингулярным уравнениям без младшего члена $\lambda ^2 u.$ Поэтому, сначала  рассмотрим некоторые свойства данного оператора.

\section{\large Обобщенный оператор Эрдейи - Кобера}\label{s2}

В теории и приложениях широко используются различные модификации и обобщения классических операторов интегрирования и дифференцирования дробного порядка Римана - Лиувилля. К таким модификациям относятся, в частности, операторы Эрдейи - Кобера \cite{Erd1}, \cite{EK}. Оказалось, что эти операторы очень полезны в приложениях к интегральным и дифференциальным уравнениям, а также в других вопросах науки и техники \cite{SKM}. Их различные модификации, обобщения и приложения могут быть найдены в работах Эрдейи~\cite{Erd2}, \cite{Erd3}, Снеддона~\cite{Snedd1}, \cite{Snedd2} и Кирьяковой~\cite{Kir}.

В работе Лоундеса~\cite{Low1} был введен и исследован обобщенный оператор Эрдейи - Кобера с функцией Бесселя в ядре
\begin{equation}\label{ooek4}
J_\lambda  (\eta ,\,\alpha )f(x) = 2^\alpha  \lambda ^{1 - \alpha } x^{ - 2\alpha  - 2\eta } \int\limits_0^x {t^{2\eta  + 1} \frac{{J_{\alpha  - 1} \left( {\lambda \sqrt {x^2  - t^2 } } \right)}}{{(x^2  - t^2 )^{(1 - \alpha )/2} }}f(t)dt},
\end{equation}
где $\alpha ,\eta ,\,\lambda  \in R,\,\,\alpha  > 0,\,\,\eta  \ge  - (1/2),$ $J_\nu  (z)$ - функция Бесселя первого рода порядка $\nu.$ Оператор~\eqref{ooek4} при $\lambda=0$  совпадает с обычным оператором  Эрдейи - Кобера \cite{SKM}
\begin{equation}\label{oek5}
I_{\eta ,\alpha } f(x) = \frac{{2x^{ - 2(\eta  + \alpha )} }}{{\Gamma (\alpha )}}\int\limits_0^x {(x^2  - t^2 )^{\alpha  - 1} t^{2\eta  + 1} f(t)dt},
\end{equation}
где $\Gamma (\alpha )$ - гамма-функция  Эйлера.

Основные свойства этих операторов можно найти в книге~\cite[c. 245, 535]{SKM}.

Свойства оператора~\eqref{ooek4} в весовых пространствах $L_p (0,\infty )$
было изучено в работах \cite{Hey1} и \cite{Hey2}. В этих работах обобщенные операторы Эрдейи-Кобера названы оператором Лоундеса.

В дальнейшем нам понадобится следующий вид оператора \eqref{ooek4}:
\begin{equation}\label{op2.2}
\,\,\,J_\lambda  (\eta ,\,\alpha )f(x) = \frac{{2x^{ - 2(\alpha  + \eta )} }}{{\Gamma (\alpha )}}\int\limits_0^x {t^{2\eta  + 1} (x^2  - t^2 )^{\alpha  - 1} \bar J_{\alpha  - 1} \left( {\lambda \sqrt {x^2  - t^2 } } \right)f(t))dt},
\end{equation}
где $\bar J_\nu  (z)$ - функция Бесселя - Клиффорда, которая выражается через функции Бесселя $J_\nu  (z)$ по формуле \cite{SKM}:
\begin{equation}\label{fbk2.3}
\bar J_\nu  (z) = \Gamma (\nu  + 1)(z/2)^{ - \nu } J_\nu  (z) = {}_0F_1 (\nu  + 1; - z^2 /4) = \sum\limits_{k = 0}^\infty  {\frac{{( - z^2 /4)^k }}{{(\nu  + 1)_k k!}}}.
\end{equation}

В работах \cite{Kar1}, \cite{Kar2} доказаны приводимые ниже теоремы.

Пусть $[B_\eta ^x ]^0  = E,$  $E-$ единичный оператор, $[B_\eta ^x ]^m  = [B_\eta ^x ]^{m - 1} [B_\eta ^x ]$ - $m-$ая степень оператора Бесселя.  В дальнейшем $m$  означает натуральное число.
\begin{theorem}\label{t2}
Пусть $\alpha  > 0, \, \eta  \geqslant  - 1/2,$ $f(x) \in C^{2m} (0,b), \, b > 0,$ функции  $x^{2\eta  + 1} [B_\eta ^x ]^{k + 1} f(x)$ интегрируемы в нуле и  $\mathop {\lim }\limits_{x \to 0} x^{2\eta  + 1} (d/dx)[B_\eta ^x ]^k f(x) = 0,$ $k = \overline {0,m - 1}. $
 Тогда
\[
[B_{\eta  + \alpha }^x  + \lambda ^2 ]^m J_\lambda  (\eta ,\alpha )f(x) = J_\lambda  (\eta ,\alpha )[B_\eta ^x ]^m f(x),
\]
в частности, если $\lambda  = 0,$ тогда
\[
[B_{\eta  + \alpha }^x ]^m I_{\eta ,\alpha } f(x) = I_{\eta ,\alpha } [B_\eta ^x ]^m f(x).
\]
\end{theorem}

Пусть функция $u(x,y) = u(x_1 ,\,x_2 ,\, \ldots ,x_n ,\,y)$  непрерывно дифференцируема до порядка $2m$ включительно по переменной $y$  и порядка не меньше чем $m$  по $x.$  $L^{(x)}$ - не зависящий от $y$ линейный дифференциальный оператор конечного порядка по переменной $x \in R^n. $
\begin{theorem}\label{t3}
Пусть $\alpha  > 0,\,\,\eta  \geqslant  - 1/2,$ функции $y^{2\eta  + 1} [B_\eta ^y ]^k u(x,y)$   интегрируемы при $y \to 0$  и $\mathop {\lim }\limits_{y \to 0} y^{2\eta  + 1} ({\partial }/{\partial y})[B_\eta ^y ]^k u(x,y) = 0,\,\,\,k = \overline {0,m - 1}. $
 Тогда
\[
(B_{\eta  + \alpha }^y  + \lambda ^2  + L^{(x)} )^m J_\lambda ^{(y)} (\eta ,\alpha )u(x,y) = J_\lambda ^{(y)} (\eta ,\alpha )(B_\eta ^y  + L^{(x)} )^m u(x,y),
\]
в частности, если $\lambda  = 0,$  тогда
\[
(B_{\eta  + \alpha }^y  + L^{(x)} )^m I_{\eta ,\alpha }^{(y)} u(x,y) = I_{\eta ,\alpha }^{(y)} (B_\eta ^y  + L^{(x)} )^m u(x,y),
\]
здесь верхние индексы в операторах означают переменные, по которым действуют эти операторы.
\end{theorem}

\begin{corolarry}\label{sl1}
Пусть $\eta  =  - 1/2,\,\,\alpha  > 0,$  функции ${\partial ^{2k} u(x,y)}/{\partial y^{2k} }$  интегрируемы при $y \to 0$  и $\mathop {\lim }\limits_{y \to 0} ({\partial ^{2k + 1} u(x,y)}/{\partial y^{2k + 1} }) = 0,\,\, k = \overline {0,m - 1}. $
 Тогда
\[
\left( {\frac{{\partial ^2 }}{{\partial y^2 }} + \frac{{2\alpha }}
{y}\frac{\partial }{{\partial y}} + \lambda ^2  + L^{(x)} } \right)^m J_\lambda ^{(y)} \left( { - \frac{1}{2},\alpha } \right)u(x,y) =
\]
\begin{equation}\label{koek2.8}
=J_\lambda ^{(y)} \left( { - \frac{1}
{2},\alpha } \right)\left( {\frac{{\partial ^2 }}{{\partial y^2 }} + L^{(x)}} \right)^m u(x,y).
\end{equation}
\end{corolarry}

		Пусть $D_\eta ^0  = E,$  $D_\eta   = x^{ - 2\eta }\left(d/(xdx)\right)x^{2\eta }, $  $D_\eta ^m  = D_\eta ^{m - 1} D_\eta ^{}  = D_\eta ^{} D_\eta ^{} ...D_\eta ^{} $ - $m$ - ая степень оператора $D_\eta ^{}, $ которая представима в виде $D_\eta ^m  = x^{ - 2\eta }\left(d/(xdx)\right)^m x^{2\eta }. $

\begin{theorem}\label{t4}
Если $\alpha  > 0,\,\,\eta  \geqslant  - 1/2,$  $f(x) \in C^m (0,b),\,\,\,b > 0,$ функции $x^{2\eta  + 1} D_\eta ^{k + 1} f(x)$  - интегрируемы в нуле и  $\mathop {\lim }\limits_{x \to 0} x^{2\eta } D_\eta ^k f(x) = 0,\,\,\,k = \overline {0,m - 1}, $  то верно равенство
\[
D_{\eta  + \alpha }^m J_\lambda  (\eta ,\alpha )f(x) = J_\lambda  (\eta ,\alpha )D_\eta ^m f(x).
\]
\end{theorem}

\begin{corolarry}\label{sl2}
Пусть $\eta  = 0,\,\,\alpha  > 0,$ $f(x) \in C^m (0,b),$ функции $\dfrac{d}{dx}\left(\dfrac{d}{xdx}\right)^k f(x)$ - интегрируемы в нуле и  $\mathop {\lim }\limits_{x \to 0}\left(\dfrac{d}{xdx}\right)^k f(x) = 0,\,\, k = \overline {0,m - 1}.$
Тогда
\[
\left( {\frac{1}
{x}\frac{d}
{{dx}}} \right)^m \int\limits_0^x {(x^2  - t^2 )^{\alpha  - 1} \bar J_{\alpha  - 1} \left( {\lambda \sqrt {x^2  - t^2 } } \right)f(t)tdt}  =
\]
\begin{equation}\label{koek1.9}
 = \int\limits_0^x {(x^2  - t^2 )^{\alpha  - 1} \bar J_{\alpha  - 1} \left( {\lambda \sqrt {x^2  - t^2 } } \right)\left[ {\left( {\frac{1}
{t}\frac{d}
{{dt}}} \right)^m f(t)} \right]tdt}
\end{equation}
в частности, если $\lambda  = 0,$ тогда
\begin{equation}\label{koek1.10}
\left( {\frac{1}{x}\frac{d}{{dx}}} \right)^m \int\limits_0^x {(x^2  - t^2 )^{\alpha  - 1} f(t)tdt}  = \int\limits_0^x {(x^2  - t^2 )^{\alpha  - 1} \left[ {\left( {\frac{1}
{t}\frac{d}{{dt}}} \right)^m f(t)} \right]tdt}.
\end{equation}
\end{corolarry}

\begin{theorem}\label{t5}
Пусть $\alpha  > 0,\,\,\eta  \geqslant  - 1/2,$ $f(x) \in C^{2m} (0,b), $ функции $x^{2\eta  + 1} [B_\eta ^x ]^{k + 1} f(x)$  интегрируемы в нуле и  $\mathop {\lim }\limits_{x \to 0} x^{2\eta  + 1}(d/dx)[B_\eta ^x ]^k f(x) = 0,$ $k = \overline {0,m - 1}. $
Тогда
\[
\frac{{d^{2m} }}{{dx^{2m} }}J_\lambda  (\eta ,\alpha )f(x) = \sum\limits_{j = 0}^m {a_{mj} x^{2j} J_\lambda  (\eta ,\alpha  + m + j)\left( {B_\eta ^x  - \lambda ^2 } \right)^{m + j} f(x)},
\]
\[
\frac{{d^{2m + 1} }}{{dx^{2m + 1} }}J_\lambda  (\eta ,\alpha )f(x) = \sum\limits_{j = 0}^m {b_{mj} x^{2j + 1} J_\lambda  (\eta ,\alpha  + m + j + 1)\left( {B_\eta ^x  - \lambda ^2 } \right)^{m + j + 1} f(x)},
\]
где  постоянные $a_{mj} $  и $b_{mj} $  определяются из следующих рекуррентных соотношений
$a_{00}  = 1,\,\,b_{00}  = 1/2,$  $b_{mj}  = (1/2)a_{mj}  + 2(j + 1)a_{m(j + 1)} ,\,\,0 \leqslant j \leqslant m,$  $a_{mj}  = 0,\,j > m,$ $a_{(m + 1)j}  = (1/2)b_{m(j - 1)}  + (2j + 1)b_{mj}, $  $1 \leqslant j \leqslant m,$ $b_{mj}  = 0,\,j > m,$ $a_{(m + 1)0}  = b_{m0}  = $ $ = 2^{ - (m + 1)} (2m + 1)!!.$
\end{theorem}

\begin{theorem}\label{t6}
Пусть $\alpha  > 0,\,\,\eta  \geqslant  - 1/2,$, $f(x) \in C^{2m - 1} [0,b] \cap C^{2m} (0,b), $ $\mathop {\lim }\limits_{x \to 0} [B_\eta ^x ]^k f(x) = c_k ,\,\,\,c_k  = const$  и  $\mathop {\lim }\limits_{x \to 0} x^{2\eta  + 1} (d/dx)[B_\eta ^x ]^k f(x) = 0,$ $k = \overline {0,m - 1}. $
 Тогда
\[
\left. {[B_{\eta  + \alpha }^x ]^m J_\lambda  (\eta ,\alpha )f(x)} \right|_{x = 0}  = \frac{{(\alpha  + \eta  + 1)_m }}{{(1/2)_m }}\left. {\frac{{d^{2m} }}
{{dx^{2m} }}J_\lambda  (\eta ,\alpha )f(x)} \right|_{x = 0},
 \]
\[
\frac{d}{{dx}}\left. {[B_{\eta  + \alpha }^x ]^m J_\lambda  (\eta ,\alpha )f(x)} \right|_{x = 0}  = 0,  \,\,
\left. {\frac{{d^{2m + 1} }}
{{dx^{2m + 1} }}J_\lambda  (\eta ,\alpha )f(x)} \right|_{x = 0}  = 0.
\]
в частности, если $\lambda  = 0,$ тогда
\[
\left. {[B_{\eta  + \alpha }^x ]^m I_{\eta ,\,\alpha } f(x)} \right|_{x = 0}  = \frac{{(\alpha  + \eta  + 1)_m }}
{{(1/2)_m }}\left. {\frac{{d^{2m} }}
{{dx^{2m} }}I_{\eta ,\,\alpha } f(x)} \right|_{x = 0}
\]
\[
\frac{d}{{dx}}\left. {[B_{\eta  + \alpha }^x ]^m I_{\eta ,\,\alpha } f(x)} \right|_{x = 0}  = 0, \,\,
\left. {\frac{{d^{2m + 1} }}
{{dx^{2m + 1} }}I_{\eta ,\,\alpha } f(x)} \right|_{x = 0}  = 0.
\]
\end{theorem}

	Доказанные теоремы позволяют сводить сингулярные (или вырождающиеся) уравнения высокого как четного, так и нечетного порядка к не сингулярным уравнениям и тем самым поставить и исследовать корректные  начальные и граничные задачи для таких уравнений.

\section{\large Приложение оператора Эрдейи-Кобера к решению поставленной задачи}\label{s3}

Предположим, что решение задачи \{\eqref{eq1}, \eqref{ic2}\} существует. Это решение будем искать в виде
\begin{equation}\label{irz11}
u(x,t) = J_\lambda ^{(t)} (\eta ,\alpha )U(x,t)
\end{equation}
где $\eta  =  - 1/2,\,\alpha  = \gamma  + 1/2 > 0,$ $U(x,t)$ -  неизвестная функция.

Подставим \eqref{irz11} в начальные  условия \eqref{ic2} и применим теоремы~\ref{t5} и \ref{t6}. Затем, подставляя в уравнение \eqref{eq1} и используя теорему~\ref{t3}, при $\eta = - 1/2,\,\, L^{(x)}  \equiv  - \Delta$  получим следующую задачу нахождения решения $U(x,t)
$  поливолнового уравнения
\begin{equation}\label{peq3.10}
L_{0,0}^m (u) \equiv \left( {\frac{{\partial ^2 }}
{{\partial t^2 }} - \Delta} \right)^m U(x,t) = 0, \,\, (x,t) \in \Omega,
\end{equation}
удовлетворяющего начальным условиям
\begin{equation}\label{peq3.11}
\left. {\frac{{\partial ^{2k} U}}{{\partial t^{2k} }}} \right|_{t = 0}  = \Phi _k (x), \,\,   \left. {\frac{{\partial ^{2k + 1} U}}{{\partial t^{2k + 1} }}} \right|_{t = 0}  = 0, \,\,  x \in R^n ,\,\,\,k = \overline {1,m - 1},
\end{equation}
где $\Phi _k (x) = \sum\limits_{j = 0}^k {a_j C_k^j \lambda ^{2(k - j)} \varphi _j (x)}, $  $a_j  = \Gamma \left( {(2j + 1)/2 + \alpha } \right)/\Gamma \left( {(2j + 1)/2} \right),$ $C_k^j  = k!/[j!(k - j)!]$ - биномиальный коэффициент, $j=\overline {0,k}, k = \overline {0,m - 1}.$

Задача Коши для поливолнового уравнения  изучалась в работах \cite{Wei2}, \cite{Kra1}, \cite{Ald2} и \cite{GK}.  Однако построенные в этих работах формулы содержат весма сложные кратные интегралы.
 Чтобы получить более простые формулы, аналогичные формулам Кирхгофа, в данной работе, задача Коши для поливолнового уравнения~\eqref{peq3.10} решается методом сферических средних.

Для полноты изложения рассмотрим задачу Коши для поливолнового уравнения \eqref{peq3.10}, удовлетворяющего начальным условиям
\begin{equation}\label{ic3.1}
\left. {\frac{{\partial ^{2k} U}}{{\partial t^{2k} }}} \right|_{t = 0}  = \Phi _k (x),\,\,\left. {\frac{{\partial ^{2k + 1} U}}{{\partial t^{2k + 1} }}} \right|_{t = 0}  = \Psi _k (x),\,\,x \in R^n ,\,\,\,k = \overline {1,m - 1},
\end{equation}
где $\Phi _k (x)$ и $\Psi _k (x),$  $(k = \overline {0,m - 1} )$ -заданные функции.

Пусть $U_0 (x,t) = U(x,t)$ и $U_k (x,t) = \left( {\dfrac{{\partial ^2 }}{{\partial t^2 }} - \Delta } \right)^k U_0 (x,t).$  Тогда задача \{\eqref{peq3.10}, \eqref{ic3.1}\} эквивалентна следующей задаче о нахождении функций $U_k (x,t),$ $k = \overline {0,m - 1},$  удовлетворяющих  системе уравнений
\begin{equation}\label{seq3.2}
\left\{ \begin{gathered}
  \frac{{\partial ^2 U_k }}
{{\partial t^2 }} - \Delta U_k  = U_{k + 1} ,\,\, (x,t) \in \Omega,\,\, k = \overline {0,m - 2} , \hfill \\
  \frac{{\partial ^2 U_{m - 1} }}
{{\partial t^2 }} - \Delta U_{m - 1}  = 0,\,\ (x,t) \in \Omega \hfill \\
\end{gathered}  \right.
\end{equation}
и начальным условиям
\begin{equation}\label{ic3.3}
U_k (x,0) = f_k (x), \;\;\; \frac{{\partial U_k (x,0)}}{{\partial t}} = g_k (x),\;\;\; x \in R^n ,\,\,\,\,\,k = \overline {0,m - 1},
\end{equation}
где
\begin{equation}\label{if3.4}
f_k (x) = \sum\limits_{j = 0}^k {( - 1)^j C_k^j \Delta^{k - j} \Phi _j (x)}, g_k (x) = \sum\limits_{j = 0}^k {( - 1)^j C_k^j \Delta^{k - j} \Psi _j (x)}, k = \overline {0,m - 1}.
\end{equation}

Для решения этой задачи применим метод сферических средних~\cite[c.694]{Kur1}.

Пусть $S(x,r) = \left\{ {\xi :\,\left| {\xi  - x} \right| = r} \right\}$ - сфера радиуса $r>0$
 с центром в точке $x \in R^n, $ где $\left| {\xi  - x} \right|^2  = \sum\limits_{k = 1}^n {(\xi _k  - x_k )^2 } $
 расстояние между точками $\xi $  и $x.$ Пусть,  далее,
\begin{equation}\label{sm3.5}
V_k (x,t,r) = \frac{1}{{\omega _n r^{n - 1} }}\int\limits_{\left| {\xi  - x} \right| = r} {U_k (\xi ,t)d\sigma _\xi  }  = \frac{1}{{\omega _n }}\int\limits_{\left| \eta  \right| = 1} {U_k (x + r\eta ,\;t)d\sigma _\eta},    \end{equation}
где  $\omega _n  = 2\pi ^{(n/2)} /\Gamma \left( {n/2} \right), $   $S(0,1) = \{ \eta :\,\left| \eta  \right| = 1\} $
- единичная сфера с центром в начале координат, $d\sigma _\xi  $ - элемент поверхности сферы $S(x,r),$ а $d\sigma _{\eta}, $ -элемент поверхности единичной сферы, причем $d\sigma _\xi   = r^{n - 1} d\sigma _{\eta}, $ $\Gamma (z)$ - гамма - функция Эйлера.

	Очевидно, что равенство~\eqref{sm3.5} является средним арифметическим значением функций $U_k (x,t)$  на сфере $S(x,r).$

Применяя классический метод сферических средних~\cite[c.694]{Kur1}, можно показать, что если функции $U_k (x,t)$ являются решением задачи Коши для системы уравнений \{\eqref{seq3.2}, \eqref{ic3.3}\}, то функции $V_k (x,t,r)$  будут решением системы уравнений
\begin{equation}\label{seq3.6}
\frac{\partial }
{{\partial r}}\left( {r^{n - 1} \frac{{\partial V_k }}
{{\partial r}}} \right) = r^{n - 1} \left( {\frac{{\partial ^2 V_k }}
{{\partial t^2 }} - V_{k + 1} } \right),\;
k = \overline {0,m - 1}; \,\,V_m (x,t) = 0, \; r > 0,\,t > 0,
\end{equation}
удовлетворяющих начальным
\begin{equation}\label{ic3.7}
V_k (x,t,0) = F_k (x,r),\,\,\,\, \left. {\frac{{\partial V_k }}
{{\partial t}}} \right|_{t = 0}  = G_k (x,r),\,\,x \in R^n ,\,\,r >0, \;\; k = \overline {0,m - 1}
\end{equation}
и  граничным
\[
\frac{{\partial V_k(x,t,0)}}
{{\partial r}} = 0,\,\,x \in R^n ,\,\,t > 0,\;\; k = \overline {0,m - 1}
\]
условиям, где
\begin{equation}\label{if3.9}
F_k (x,r) = \frac{1}{{\omega _n r^{n - 1} }}\int\limits_{\left| {\xi  - x} \right| = r} {f_k (\xi )d\sigma _\xi  }  = \frac{1}{{\omega _n }}\int\limits_{\left| \eta  \right| = 1} {f_k (x + r\eta )d\sigma _\eta},
\end{equation}
\begin{equation}\label{if3.10}
G_k (x,r) = \frac{1}{{\omega _n r^{n - 1} }}\int\limits_{\left| {\xi  - x} \right| = r} {g(\xi )d\sigma _\xi  }  = \frac{1}{{\omega _n }}\int\limits_{\left| \eta  \right| = 1} {g(x + r\eta )d\sigma _\eta},
\end{equation}
причем $\dfrac{{\partial F_k }}{{\partial r}}(x,0) = 0,\,\,\,\dfrac{{\partial G_k }}{{\partial r}}\,(x,0) = 0,\;\; k = \overline {0,m - 1}.$

Каждое уравнение системы~\eqref{seq3.6} перепишем в виде
$$
\frac{1}
{r}\frac{\partial }
{{\partial r}}\left( {r^{n - 1} \frac{{\partial V_k }}
{{\partial r}}} \right) - \frac{{\partial ^2 }}
{{\partial t^2 }}\left( {r^{n - 2} V_k } \right) + \left( {r^{n - 2} V_{k + 1} } \right) = 0.
$$

	Пусть $n$  - нечетное натуральное число:  $n = 2p + 1,$ $p \in N.$  Тогда, применяя к последнему равенству  оператор $\left( {\dfrac{1}{r}\dfrac{\partial }{{\partial r}}} \right)^{p - 1}, $ получим
\begin{equation}\label{seq3.11}
\left( {\frac{1}
{r}\frac{\partial }
{{\partial r}}} \right)^p \left( {r^{2p} \frac{{\partial V_k }}
{{\partial r}}} \right) - \frac{{\partial ^2 }}
{{\partial t^2 }}\left[ {\left( {\frac{1}
{r}\frac{\partial }
{{\partial r}}} \right)^{p - 1} \left( {r^{2p - 1} V_k } \right)} \right] + \left( {\frac{1}
{r}\frac{\partial }
{{\partial r}}} \right)^{p - 1} \left( {r^{2p - 1} V_{k + 1} } \right) = 0.
\end{equation}

Далее, воспользуемся следующей леммой~\cite{Eva1}:
\begin{lemma}\label{lem1}
Если $w(r) \in C^{p + 1} ,\,\,p = 1,\,2,\, \ldots, $ то справедливы равенства
\begin{equation}\label{leq3.12}
\frac{{d^2 }}{{dr^2 }}\left( {\frac{1}{r}\frac{d}{{dr}}} \right)^{p - 1} \left( {r^{2p - 1} w} \right) = \left( {\frac{1} {r}\frac{d}{{dr}}} \right)^p \left( {r^{2p} \frac{{dw}}{{dr}}} \right),
\end{equation}
\begin{equation}\label{leq3.13}
\left( {\frac{1}{r}\frac{d}	{{dr}}} \right)^{p - 1} \left( {r^{2p - 1} w} \right) = \sum\limits_{j = 0}^{p - 1} {A_j^p r^{j + 1} \frac{{d^j w}}{{dr^j }}},
\end{equation}
где $A_j^p  = const,$  причем $A_0^p  = 1 \cdot 3 \cdot 5 \cdot 7 \cdots (2p - 1) = (2p - 1)!!.$
\end{lemma}

Введя обозначение
\begin{equation}\label{uf3.14}
W_k (x,t,r) = \left( {\frac{1}{r}\frac{\partial }{{\partial r}}} \right)^{p - 1} \left( {r^{2p - 1} V_k } \right),\,k = \overline {0,m - 1}
\end{equation}
и учитывая лемму~\ref{lem1}, из  равенства~\eqref{seq3.11}, \eqref{leq3.12} и начальных условий~\eqref{ic3.7}  получим задачу о нахождении функций $W_k (x,t,r), \;\;k=\overline {0,m - 1},$  удовлетворяющих системе уравнений
\begin{equation}\label{seq3.15}
\frac{{\partial ^2 W_k }}
{{\partial t^2 }} - \frac{{\partial ^2 W_k }}
{{\partial r^2 }} = W_{k + 1} ,\,\,\,\,k = \overline {0,m - 1};\;\;\; W_m (x,t,r) = 0,\,\, r > 0,\,\,\,t > 0,    \end{equation}
а также начальным
\begin{equation}\label{ic3.16}
W_k (x,0,r) = \Phi _k (x,r),\;\;\; \frac{{\partial W_k (x,0,r)}}
{{\partial t}} = \Psi _k (x,r),\,\,\,x \in R^n ,\,\,r > 0, \;\; k = \overline {0,m - 1}
\end{equation}
и граничным
\begin{equation}\label{bc3.17}
W_k (x,t,0) = 0,\,\,\,x \in R^n ,\,\,\,t > 0, \;\; k = \overline {0,m - 1}
\end{equation}
условиям, где  	
\begin{equation}\label{if3.18}
\Phi _k (x,r) = \left( {\frac{1}{r}\frac{\partial }{{\partial r}}} \right)^{p - 1} \left( {r^{2p - 1} F_k (x,r)} \right),
\end{equation}
\begin{equation}\label{if3.18a}
\Psi _k (x,r) = \left( {\frac{1}{r}\frac{\partial }
{{\partial r}}} \right)^{p - 1} \left( {r^{2p - 1} G_k (x,r)} \right).
\end{equation}

Учитывая граничные условия~\eqref{bc3.17}, начальные функции  $\Phi _k (x,r)$ и  $\Psi _k (x,r)$  ($k = \overline {0,m - 1}$) продолжим нечетным образом на полуинтервал $r < 0$ и продолженные функции обозначим через $\tilde \Phi _k (x,r)$  и $\tilde \Psi _k (x,r)$  соответственно.

	Тогда, в области $\Omega _1  = \{ (r,t):\, - \infty  < r <  + \infty ,\,0 < t <  + \infty \} $ получим задачу нахождения решений  $W_k (x,t,r)$ системы уравнений~\eqref{seq3.15},  удовлетворяющих начальным условиям
\begin{equation}\label{ic3.19}
W_k (x,0,r) = \tilde \Phi _k (x,r),\;\;\; \frac{{\partial W_k (x,0,r)}}
{{\partial t}} = \tilde \Psi _k (x,r),\,\,\,x \in R^n ,\,\,\,r \in R^1, \;\; k = \overline {0,m - 1}.    \end{equation}

Решение задачи \{\eqref{seq3.15}, \eqref{ic3.19}\} имеет вид \cite{Kar3}:
$$
W_0 (x,t,r) = \frac{1}
{2}\left[ {\tilde \Phi _0 (x,r + t) + \tilde \Phi _0 (x,r - t)} \right] + \frac{1}
{2}\int\limits_{r - t}^{r + t} {\tilde \Psi _0 (x,s)ds}  +
$$
$$
 + \sum\limits_{k = 1}^{m - 1} {\frac{1}{{2^{2k + 1} (k!)^2 }}} \frac{\partial }{{\partial t}}\int\limits_{r - t}^{r + t} {\left[ {t^2  - (r - s)^2 } \right]^k \tilde \Phi _k (x,s)ds + }
$$
\begin{equation}\label{ss3.23}
 + \sum\limits_{k = 1}^{m - 1} {\frac{1}{{2^{2k + 1} (k!)^2 }}} \int\limits_{r - t}^{r + t} {\left[ {t^2  - (r - s)^2 } \right]^k \tilde \Psi _k (x,s)ds}.
\end{equation}

Для нахождения решения задачи Коши \{\eqref{peq3.10}, \eqref{ic3.1}\} воспользуемся следующим свойством среднее сферических~\cite[c. 694]{Kur1}:
\begin{equation}\label{sc3.24}
U(x,t) = U_0 (x,t) = \mathop {\lim }\limits_{r \to 0} V_0 (x,t,r).
\end{equation}

	Применяя  \eqref{leq3.13}, равенство \eqref{uf3.14} можно представить в виде
\[
V_0 (x,t,r) = \frac{{W_0 (x,t,r)}}{{A_0^p r}} - o(r).
\]

	Учитывая это,  из \eqref{sc3.24} получим
\begin{equation}\label{sc3.26}
U(x,t) = \mathop {\lim }\limits_{r \to 0} V_0 (x,t,r) = \frac{1}
{{A_0^p }}\mathop {\lim }\limits_{r \to 0} \frac{{W_0 (x,t,r)}}{r}.
\end{equation}

	В силу нечетности функций $\tilde \Phi _k (x,r)$ и $\tilde \Psi _k (x,r)$  относительно переменной $r,$
 из \eqref{ss3.23} следует, что $W_0 (x,t,0) = 0.$

  Для вычисления предела \eqref{sc3.26} применим правило Лопиталля~\cite[c.270]{IP1}. \, Тогда,  после  несложных преобразований, имеем
$$
U(x,t) = \frac{1}
{{A_0^p }}\mathop {\lim }\limits_{r \to 0} \frac{{\partial W_0 }}
{{\partial r}} = \frac{1}
{{A_0^p }}\left[ {\frac{\partial }
{{\partial t}}\Phi _0 (x,t) + \Psi _0 (x,t)} \right] +
$$
$$
+ \frac{1}{{A_0^p }}\sum\limits_{k = 1}^{m - 1} {\frac{1}{{2^{2k - 1} (k - 1)!k!}}} \frac{\partial }{{\partial t}}\int\limits_0^t {\left( {t^2  - s^2 } \right)^{k - 1} s\Phi _k (x,s)ds + }
$$
\begin{equation}\label{sc3.27}
+ \frac{1}{{A_0^p }}\sum\limits_{k = 1}^{m - 1} {\frac{1}{{2^{2k - 1} (k - 1)!k!}}} \int\limits_0^t {\left( {t^2  - s^2 } \right)^{k - 1} s\Psi _k (x,s)ds}.
\end{equation}

В силу \eqref{leq3.13} для функций  $s^{2p - 1} F_k (x,s)$ и  $s^{2p - 1}G_k (x,s)$  условия следствия 2  выполняются. Поэтому, применяя формулу~\eqref{koek1.10} при $\alpha  = k > 0$ к равенству~\eqref{sc3.27}, с учетом \eqref{if3.18} и \eqref{if3.18a} получим
$$
U(x,t) = \frac{1}{{A_0^p }}\left[ {\frac{\partial }{{\partial t}}\left( {\frac{1}{t}\frac{\partial }{{\partial t}}} \right)^{p - 1} \left( {t^{2p - 1} F_0 (x,t)} \right) + \left( {\frac{1}{t}\frac{\partial }{{\partial t}}} \right)^{p - 1} \left( {t^{2p - 1} G_0 (x,t)} \right)} \right] +
$$
$$
 + \frac{1}{{A_0^p }}\sum\limits_{k = 1}^{m - 1} {\frac{1}{{2^{2k - 1} (k - 1)!k!}}\frac{\partial }{{\partial t}}} \left( {\frac{1}{t}\frac{\partial }{{\partial t}}} \right)^{p - 1} \int\limits_0^t {\left( {t^2  - s^2 } \right)^{k - 1} s^{2p} F_k (x,s)ds + }
$$
\begin{equation}\label{sc3.28}
 + \frac{1}{{A_0^p }}\sum\limits_{k = 1}^{m - 1} {\frac{1}{{2^{2k - 1} (k - 1)!k!}}} \left( {\frac{1}{t}\frac{\partial }{{\partial t}}} \right)^{p - 1} \int\limits_0^t {\left( {t^2  - s^2 } \right)^{k - 1} s^{2p} G_k (x,s)ds}.
\end{equation}

Подставляя \eqref{if3.9} и \eqref{if3.10} в равенство \eqref{sc3.28}, имеем
$$
U(x,t) =\gamma _n \left[ {\frac{\partial }{{\partial t}}\left( {\frac{1}{t}\frac{\partial }
{{\partial t}}} \right)^{\frac{{n - 3}}{2}} \left( {\frac{1}{t}\int\limits_{\left| {\xi  - x} \right| = t} {f_0 (\xi )d\sigma _\xi  } } \right)}+\right.
$$
$$
+\left. {\left( {\frac{1}{t}\frac{\partial }{{\partial t}}} \right)^{\frac{{n - 3}}
{2}} \left( {\frac{1}{t}\int\limits_{\left| {\xi  - x} \right| = t} {g_0 (\xi )d\sigma _\xi  } } \right)} \right] +
$$
$$
+ \gamma _n \sum\limits_{k = 1}^{m - 1} {\frac{1}{{2^{2k - 1} (k - 1)!k!}}} \frac{\partial }{{\partial t}}\left( {\frac{1}{t}\frac{\partial }{{\partial t}}} \right)^{\frac{{n - 3}}{2}} \int\limits_{\left| {\xi  - x} \right| \le t} {\left[ {t^2  - \left| {\xi  - x} \right|^2 } \right]^{k - 1} f_k (\xi )d\xi }  +
$$
\begin{equation}\label{sc3.29}
+ \gamma _n \sum\limits_{k = 1}^{m - 1} {\frac{1}{{2^{2k - 1} (k - 1)!k!}}} \left( {\frac{1}{t}\frac{\partial }{{\partial t}}} \right)^{\frac{{n - 3}}{2}} \int\limits_{\left| {\xi  - x} \right| \le t} {\left[ {t^2  - \left| {\xi  - x} \right|^2 } \right]^{k - 1} g_k (\xi )d\xi},
\end{equation}
где $\gamma _n  = [1 \cdot 3 \cdot 5 \cdot ... \cdot (n - 2)\omega _n ]^{ - 1},$ $f_k (x),$ $g_k (x),$ $k = \overline {0,m - 1} $ - функции определяемые через заданные начальные данные равенствами~\eqref{if3.4}.

Как и в случае волнового уравнения \cite{Kur1}, можно показать, что если $\varphi _j (x) \in C^{q_1 } (\Omega _n ),$ $\psi _j (x) \in C^{q_2 } (\Omega _n ),$ $q_1  = [(n + 1)/2] + 2(m - j) + 2,$  $q_2  = [(n + 1)/2] + 2(m - j) + 1,$ $j = \overline {0,m - 1}, $   где $[(n + 1)/2]$ -означает целую часть числа $(n + 1)/2,$ то функция $U(x,t),$ определяемая равенством~\eqref{sc3.29},   при нечетном $n$  является классическим решением задачи Коши \{\eqref{peq3.10}, \eqref{ic3.1}\},  из которого при $m = 1$  следует формула Кирхгофа для волнового уравнения.

При четном $n,$ применяя метод спуска Адамара \cite[c.682]{Kur1}, \cite[c.74]{Eva1}, из формулы \eqref{sc3.29} также можно получить явную формулу решения задачи Коши \{\eqref{peq3.10}, \eqref{ic3.1}\}, которая имеет вид:
$$
U(x,t) = \tilde \gamma _n \sum\limits_{k = 0}^{m - 1} {\frac{{\Gamma ^{ - 1} (k + 1/2)}}{{2^{2k} k!}}} \frac{\partial }
{{\partial t}}\left( {\frac{1}
{t}\frac{\partial }
{{\partial t}}} \right)^{\frac{{n - 2}}
{2}} \int\limits_{\left| {\xi  - x} \right| < t} {\left[ {t^2  - \left| {\xi  - x} \right|^2 } \right]^{k - (1/2)} f_k (\xi )d\xi }  +
$$
\begin{equation}\label{sc3.30}
 + \tilde \gamma _n \sum\limits_{k = 0}^{m - 1} {\frac{{\Gamma ^{ - 1} (k + 1/2)}}{{2^{2k} k!}}} \left( {\frac{1}
{t}\frac{\partial }
{{\partial t}}} \right)^{\frac{{n - 2}}
{2}} \int\limits_{\left| {\xi  - x} \right| < t} {\left[ {t^2  - \left| {\xi  - x} \right|^2 } \right]^{k - (1/2)} g_k (\xi )d\xi },
\end{equation}
где  $\tilde \gamma _n  = [\omega _{n + 1} 2 \cdot 4 \cdot 6 \cdot ... \cdot (n - 1)]^{ - 1}.$

Аналогично, как и в случае нечетного $n,$ можно показать, что если $\varphi _j (x) \in C^{q_1 } (\Omega _n ),$ $\psi _j (x) \in C^{q_2 } (\Omega _n ),$ $q_1  = [n/2] + 2(m - j) + 2,$ $q_2  = [n/2] + 2(m - j) + 1,$ $j = \overline {0,m - 1}, $ то функция $U(x,t),$  определяемая равенством~\eqref{sc3.30},   при четном $n$  является классическим решением задачи Коши \{\eqref{peq3.10}, \eqref{ic3.1}\}.

\section{\large Решение основной задачи}\label{s4}

	Вернемся к исследованию задачи \{\eqref{eq1}, \eqref{ic2}\}. Учитывая решение вспомогательной задачи  \{\eqref{peq3.10}, \eqref{ic3.1}\} при нечетном $n = 2p + 1,\,p \in N$  и $\Psi _k (x) = 0,$ из \eqref{sc3.28}, имеем
\[
U(x,t) = \frac{1}
{{A_0^p }}\frac{\partial }
{{\partial t}}\left( {\frac{1}
{t}\frac{\partial }
{{\partial t}}} \right)^{p - 1} \left( {t^{2p - 1} F_0 (x,t)} \right) +
\]
\begin{equation}\label{sc39}
+ \frac{1}
{{A_0^p }}\sum\limits_{k = 1}^{m - 1} {\frac{{2^{ - 2k+1} }}
{{(k - 1)!k!}}} \frac{\partial }
{{\partial t}}\left( {\frac{1}
{t}\frac{\partial }
{{\partial t}}} \right)^{p - 1} \int\limits_0^t {(t^2  - \tau ^2 )^{k - 1} \tau ^{2p} F_k (x,\tau )d\tau },
\end{equation}
где  функции $F_k (x,\tau )$  определяются равенством \eqref{if3.9}, $A_0^p  = 1 \cdot 3 \cdot 5 \cdot 7 \cdots (2p - 1) = (2p - 1)!!.$

Подставляя \eqref{sc39} в  \eqref{irz11},  после несложных преобразований, получим
\begin{equation}\label{sc4.2}
u(x,t) = \frac{{t^{1 - 2\alpha }}}{{A^p_0 \Gamma (\alpha )}} Q_0 (x,t) + \frac{{t^{1 - 2\alpha } }}{{A^p_0 \Gamma (\alpha )}} \sum\limits_{k = 1}^{m - 1} {\frac{{2^{ - 2k} }}{{(k - 1)!k!}}} Q_k (x,t),
\end{equation}
где
\begin{equation*}\label{sc4.3}
Q_0 (x,t) = \int\limits_0^t {(t^2  - s^2 )^{\alpha  - 1} \bar J_{\alpha  - 1} \left( {\lambda \sqrt {t^2  - s^2 } } \right)\left[ {\frac{\partial }{{\partial s}}\left( {\frac{1}{s}\frac{\partial }{{\partial s}}} \right)^{p - 1} \left( {s^{2p - 1} F_0 (x,s)} \right)} \right]ds},
\end{equation*}
$$
Q_k (x,t) = \int\limits_0^t {(t^2  - s^2 )^{\alpha  - 1} \bar J_{\alpha  - 1} \left( {\lambda \sqrt {t^2  - s^2 } } \right) \times}
$$
\begin{equation*}\label{sc4.4}
\times \left[ {\frac{\partial }{{\partial s}}\left( {\frac{1}{s}\frac{\partial }{{\partial s}}} \right)^{p - 1} \int\limits_0^s {(s^2  - \tau ^2 )^{k - 1} \tau ^{2p} F_k (x,\tau )d\tau } } \right]ds.
\end{equation*}

В силу \eqref{leq3.13} для функций $s^{2p - 1} F_0 (x,s)$  и  $s^{2p - 1} \tilde F_k (x,s),$ где
\[
\tilde F_k (x,s) = s^{2k} \int\limits_0^1 {(1 - z^2 )^{k - 1} z^{2p} F_k (x,sz)dz},
\]
условия следствия 2 выполняются. Поэтому применяя формулу \eqref{koek1.9}, имеем
\[
Q_0 (x,t) = \left( {\frac{1}{t}\frac{\partial }{{\partial t}}} \right)^p \int\limits_0^t {(t^2  - s^2 )^{\alpha  - 1} \bar J_{\alpha  - 1} \left( {\lambda \sqrt {t^2  - s^2 } } \right)s^{2p} F_0 (x,s)ds},
\]
\[	
Q_k (x,t) = \left( {\frac{1}{t}\frac{\partial }{{\partial t}}} \right)^p \int\limits_0^t {(t^2  - s^2 )^{\alpha  - 1} \bar J_{\alpha  - 1} \left( {\lambda \sqrt {t^2  - s^2 } } \right)s \times}
\]
\begin{equation}\label{sc41}
\times \left[ {\int\limits_0^s {(s^2  - \tau ^2 )^{k - 1} \tau ^{2p} F_k (x,\tau )d\tau } } \right]ds \end{equation}

	Учитывая вид \eqref{if3.9} функции  $F_0 (x,s),$ функцию $Q_0 (x,t)$  можно написать в виде
\begin{equation}\label{sc42}
Q_0 (x,t) = \frac{1}{{\omega _n }}\left( {\frac{1}{t}\frac{\partial }
{{\partial t}}} \right)^{\frac{{n - 1}}{2}} \int\limits_{\left| {\xi  - x} \right| < t} {\left[ {t^2  - \left| {\xi  - x} \right|^2 } \right]^{\alpha  - 1} \bar J_{\alpha  - 1} \left( {\lambda \sqrt {t^2  - \left| {\xi  - x} \right|^2 } } \right)f_0 (\xi )d\xi }
\end{equation}

 Произведя перестановку порядка интегрирования, из \eqref{sc41} получим
\begin{equation}\label{sc43}
Q_k (x,t) = \left( {\frac{1}{t}\frac{\partial }{{\partial t}}} \right)^p \int\limits_0^t {\tau ^{2p} F_k (x,\tau )K(t,\tau )d\tau },
\end{equation}
где
\[
K(t,\tau ) = \int\limits_\tau ^t {(t^2  - s^2 )^{\alpha  - 1} \bar J_{\alpha  - 1} \left( {\lambda \sqrt {t^2  - s^2 } } \right)s} (s^2  - \tau ^2 )^{k - 1} ds.
\]

Вычислим последний интеграл. Произведя замену переменных интегрирования по формуле $s^2  = \tau ^2  + [t^2  - \tau ^2 ]\mu, $ получим
\[
K(t,\tau ) = (1/2)(t^2  - \tau ^2 )^{\alpha  + k - 1} \int\limits_0^1 {\mu ^{k - 1} (1 - \mu )^{\alpha  - 1} \bar J_{\alpha  - 1} \left( {\lambda \sqrt {(t^2  - \tau ^2 )(1 - \mu )} } \right)d\mu }.
\]

В правой части последнего равенства, пользуясь разложением функции Бесселя - Клиффорда в ряд \eqref{fbk2.3} и  равномерной сходимостью данного ряда при любых значениях аргумента, меняем порядок интегрирования и суммирования. Затем, вычислив полученный интеграл, имеем
\[
K(t,\tau ) = \frac{{\Gamma (k)\Gamma (\alpha )}}
{{2\Gamma (\alpha  + k)}}(t^2  - \tau ^2 )^{\alpha  + k - 1} \bar J_{\alpha  + k - 1} \left( {\lambda \sqrt {t^2  - \tau ^2 } } \right).
\]

Принимая во внимание последнее равенство, из \eqref{sc43} получим
\begin{equation}\label{sc44}
Q_k (x,t) = \frac{{\Gamma (k)\Gamma (\alpha )}}
{{2\Gamma (\alpha  + k)}}\left( {\frac{1}
{t}\frac{\partial }
{{\partial t}}} \right)^p \int\limits_0^t {\tau ^{2p} F_k (x,\tau )(t^2  - \tau ^2 )^{\alpha  + k - 1} \bar J_{\alpha  + k - 1} \left( {\lambda \sqrt {t^2  - \tau ^2 } } \right)d\tau }.
\end{equation}

Подставляя  в \eqref{sc44} выражение \eqref{if3.9} функций $F_k (x,s),$ находим
\[
Q_k (x,t) = \frac{{\Gamma (k)\Gamma (\alpha )}}
{{2 \omega_n \Gamma (\alpha  + k)}}\left( {\frac{1}
{t}\frac{\partial }
{{\partial t}}} \right)^p  \times
\]
\begin{equation}\label{sc45}
 \times \int\limits_{\left| {\xi  - x} \right| < t} {\left[ {t^2  - \left| {\xi  - x} \right|^2 } \right]^{\alpha  + k - 1} \bar J_{\alpha  + k - 1} \left( {\lambda \sqrt {t^2  - \left| {\xi  - x} \right|^2 } } \right)f_k (\xi )d\xi }.
\end{equation}

Подставив \eqref{sc42} и \eqref{sc45} в \eqref{sc4.2} и принимая во внимание $\Gamma (k) = (k - 1)!,$ окончательно находим явную формулу решения задачи \{\eqref{eq1}, \eqref{ic2}\} при нечетном $n:$
\[
u(x,t) = \bar \gamma _n t^{1 - 2\alpha } \left( {\frac{1}
{t}\frac{\partial }
{{\partial t}}} \right)^{\frac{{n - 1}}
{2}} \int\limits_{\left| {\xi  - x} \right| < t} {\frac{\bar J_{\alpha  - 1} \left( {\lambda \sqrt {t^2  - \left| {\xi  - x} \right|^2 } } \right)}{\left[ {t^2  - \left| {\xi  - x} \right|^2 } \right]^{1-\alpha}} f_0 (\xi )d\xi }  +
\]
\begin{equation}\label{sc46}
 + \bar \gamma _n t^{1 - 2a} \sum\limits_{k = 1}^{m - 1} {\frac{{2^{ - 2k} }}
{{k!(\alpha )_k }}} \left( {\frac{1}
{t}\frac{\partial }
{{\partial t}}} \right)^{\frac{{n - 1}}
{2}} \int\limits_{\left| {\xi  - x} \right| < t} { \frac{\bar J_{\alpha  + k - 1} \left( {\lambda \sqrt {t^2  - \left| {\xi  - x} \right|^2 } } \right)}{\left[ {t^2  - \left| {\xi  - x} \right|^2 } \right]^{1-\alpha - k} }  f_k (\xi )d\xi },
\end{equation}
где  $\bar\gamma _n  = [1 \cdot 3 \cdot 5 \cdot ... \cdot (n - 2)\omega _n \Gamma (\alpha )]^{ - 1}. $

	При четном $n,$ применяя метод спуска Адамара \cite[c.682]{Kur1}, \cite[c.74]{Eva1}, из формулы \eqref{sc46} также можно получить явную формулу решения задачи Коши \{\eqref{eq1}, \eqref{ic2}\}, которое имеет вид
\[
u(x,t) = \bar \gamma _n t^{1 - 2\alpha } \sum\limits_{k = 0}^{m - 1} {\frac{{2^{ - 2k} }}
{{k!(\alpha  + 1/2)_k }}}  \times
\]
\begin{equation}\label{sc47}
 \times \left( {\frac{1}
{t}\frac{\partial }
{{\partial t}}} \right)^{\frac{n}
{2}} \int\limits_{\left| {\xi  - x} \right| < t} {\left[ {t^2  - \left| {\xi  - x} \right|^2 } \right]^{\alpha  + k - (1/2)} \bar J_{\alpha  + k - (1/2)} \left( {\lambda \sqrt {t^2  - \left| {\xi  - x} \right|^2 } } \right)f_k (\xi )d\xi }.
\end{equation}

Из теоремы~\ref{t6} следует, что в задаче \{\eqref{eq1}, \eqref{ic2}\} вместо начального условия \eqref{ic2} можно взять начальные условия вида
\begin{equation}\label{sc48}
\left. {[B_{\alpha  - (1/2)}^t ]^k u} \right|_{t = 0}  = \varphi _k^* (x),\,\,x \geqslant 0; \,\,
\left. {\frac{\partial }
{{\partial t}}[B_{\alpha  - (1/2)}^t ]^k u} \right|_{t = 0}  = 0,\,\,x > 0,\,\,\,k = \overline {0,m - 1},
\end{equation}
где  $\varphi _k^* (x) = [(\alpha  + 1/2)_k /(1/2)_k ]\varphi _k (x).$

Кроме того, в силу теоремы~\ref{t2} и \ref{t6}, а также равенства \eqref{irz11} следует, что задачи \{\eqref{eq1}, \eqref{ic2}\} и \{\eqref{eq1}, \eqref{sc48}\} сводятся к одной и той же вспомогательной задаче \{\eqref{peq3.10}, \eqref{peq3.11}\}. Отсюда следует справедливость следующего утверждения:
\begin{lemma}\label{lem2}
 Решение задачи \{\eqref{eq1}, \eqref{ic2}\}  является решением задачи   \{\eqref{eq1}, \eqref{sc48}\}, и наоборот.
\end{lemma}

Данная лемма при $\lambda=0$ в работах \cite{Iva1} и \cite{Ald1} доказана другим методом. Аналогичное утверждение имеет место и для задачи \{\eqref{eq1}, \eqref{ic3}\} . В работах \cite{Iva1}, \cite{Ald1} при $\lambda=0$ доказано, что вместо начальных условий \eqref{ic3} можно взять начальные условия вида
\begin{equation}\label{sc49}
\left. {[B_{\alpha  - (1/2)}^t ]^k u} \right|_{t = 0}  = 0,\,\,x \geqslant 0;
\left. {t^{2\alpha} \frac{\partial }
{{\partial t}}[B_{\alpha  - (1/2)}^t ]^k u} \right|_{t = 0}  = \psi _k^* (x),\,\,x > 0,\,\,\,k = \overline {0,m - 1},
\end{equation}
где  $\psi _k (x) = \prod\limits_{j = 1}^k {(1 - (\alpha /j))} \psi _k^* (x).$

Аналогично можно доказать справедливость следующего утверждения
\begin{lemma}\label{lem3}
При любом $\alpha  < 1/2$ решение задачи \{\eqref{eq1}, \eqref{ic3}\} является решением задачи \{\eqref{eq1}, \eqref{sc49}\}, и наоборот.
\end{lemma}

Теперь рассмотрим задачу нахождения решения $u_2 (x,t)$ уравнения \eqref{eq1}, удовлетворяющего начальным условиям \eqref{sc49}.

\begin{lemma}\label{lem4}
Если $u_1 (x,t;1 - \alpha )$ является решением уравнения $ L^m_{1 - \alpha, \lambda }(u_1 ) = 0,$ удовлетворяющего условиям \eqref{sc48}, в котором $\alpha $ заменяется на $1 - \alpha, $
то функция $u_2 (x,t;\alpha ) = t^{1 - 2\alpha } u_1 (x,t;1 - \alpha )$ при $0 < \alpha  < (1/2)$
будет решением уравнения  $L^m_{\alpha, \lambda }(u_2 ) = 0,$ удовлетворяющего условиям
\[
\left. {[B_{\alpha  - (1/2)}^t ]^k u_2 } \right|_{t = 0}  = 0,\,\,x \geqslant 0; \,\,
\left. {t^{2\alpha } \frac{\partial }{{\partial t}}[B_{\alpha  - (1/2)}^t ]^k u_2 } \right|_{t = 0}  = (1 - 2\alpha )\varphi _k^* (x),\,\,x > 0, \,\, k = \overline {0,m - 1}.
\]
\end{lemma}

\begin{proof}
Методом математической индукции по $m,$ докажем справедливость следующих равенств
\begin{equation}\label{pl51}
L_{\alpha ,\lambda }^m (u_2 ) = L_{\alpha ,\lambda }^m (t^{1 - 2\alpha } u_1 ) = t^{1 - 2\alpha } L_{1 - \alpha, \lambda }^m (u_1 ),
\end{equation}
\begin{equation}\label{pl52}
[B_{\alpha  - (1/2)}^t ]^m u_2  = [B_{\alpha  - (1/2)}^t ]^m t^{1 - 2\alpha } u_1  = t^{1 - 2\alpha } [B_{(1/2) - \alpha }^t ]^m u_1.
\end{equation}

	При $m = 1$  справедливость формулы~\eqref{pl51} следует из известной формулы соответствий \cite[c.577]{SKM}: $L_{\alpha ,\lambda }^m (t^{1 - 2\alpha } u_1 ) = t^{1 - 2\alpha } L_{1 - \alpha ,\lambda }^m (u_1 ),$ а формула~\eqref{pl52} проверяется непосредственным вычислением
\[
[B_{\alpha  - (1/2)}^t ]t^{1 - 2\alpha } u_1  = t^{ - 2\alpha } \frac{\partial }
{{\partial t}}t^{2\alpha } \frac{\partial }
{{\partial t}}t^{1 - 2\alpha } u_1  = t^{ - 2\alpha } \frac{\partial }
{{\partial t}}\left[ {t\frac{{\partial u_1 }}
{{\partial t}} + (1 - 2\alpha )u_1 } \right] =
\]
\[
 = t^{1 - 2\alpha } \left[ {\frac{{\partial ^2 u_1 }}
{{\partial t^2 }} + \frac{{2(1 - \alpha )}}
{t}\frac{{\partial u_1 }}
{{\partial t}}} \right] = t^{1 - 2\alpha } [B_{(1/2) - \alpha }^t ]u_1.
\]

	Предположим, что формулы \eqref{pl51} и \eqref{pl52} верны при $m = p.$ Докажем их справедливость при $m = p + 1.$

	Из равенства $L_{\alpha ,\lambda }^{p + 1} (t^{1 - 2\alpha } u_1 ) = L_{\alpha ,\lambda }^{} [L_{\alpha ,\lambda }^p (t^{1 - 2\alpha } u_1 )]$  в силу предположения индукции, имеем $L_{\alpha ,\lambda }^{p + 1} (t^{1 - 2\alpha } u_1 ) = L_{\alpha ,\lambda }^{} [t^{1 - 2\alpha } L_{1 - \alpha ,\lambda }^p (u_1 )].$
Далее, применяя формулы соответствий, получим  $L_{\alpha ,\lambda }^{p + 1} (t^{1 - 2\alpha } u_1 ) = t^{1 - 2\alpha } L_{1 - \alpha ,\lambda }^{p + 1} (u_1 ).$
Из последнего равенства следует справедливость формулы \eqref{pl51}. Аналогично доказывается равенство \eqref{pl52}.

В равенстве \eqref{pl52} вычислив производную по $t,$ имеем
\[
\frac{\partial }{{\partial t}}[B_{\alpha  - (1/2)}^t ]^k u_2  = \frac{\partial }
{{\partial t}}\left\{ {t^{1 - 2\alpha } [B_{(1/2) - \alpha }^t ]^k u_1 } \right\} =
\]
\[
= (1 - 2\alpha )t^{ - 2\alpha } [B_{(1/2) - \alpha }^t ]^k u_1  + t^{1 - 2\alpha } \frac{\partial }
{{\partial t}}[B_{(1/2) - \alpha }^t ]^k u_1.
\]

	Отсюда, в силу условий \eqref{sc49} в котором $\alpha $  заменено на $1 - \alpha, $ следует, что
\[
\left. {t^{2\alpha } \frac{\partial }
{{\partial t}}[B_{\alpha  - (1/2)}^t ]^k u_2 } \right|_{t = 0}  = \left[ {(1 - 2\alpha )[B_{(1/2) - \alpha }^t ]^k u_1  + t\frac{\partial }{{\partial t}}[B_{(1/2) - \alpha }^t ]^k u_1 } \right]_{t = 0}  = (1 - 2\alpha )\varphi _k^* (x).
\]
Этим завершается доказательство леммы 4.

\end{proof}

Для решения задачи \{\eqref{eq1}, \eqref{sc49}\}, применяя лемму 4, в силу \eqref{sc46} при нечетном $n,$ имеем
\[
u_2 (x,t) = \bar \gamma _n \left( {\frac{1}{t}\frac{\partial }
{{\partial t}}} \right)^{\frac{{n - 1}}{2}} \int\limits_{\left| {\xi  - x} \right| < t} {\left[ {t^2  - \left| {\xi  - x} \right|^2 } \right]^{ - \alpha } \bar J_{ - \alpha } \left( {\lambda \sqrt {t^2  - \left| {\xi  - x} \right|^2 } } \right)g_0^* (\xi )d\xi }  +
\]
\[
 + \bar \gamma _n \sum\limits_{k = 1}^{m - 1} {\frac{{2^{ - 2k} }}
{{k!(1 - \alpha )_k }}} \left( {\frac{1}
{t}\frac{\partial }
{{\partial t}}} \right)^{\frac{{n - 1}}
{2}} \int\limits_{\left| {\xi  - x} \right| < t} {\frac{{\bar J_{ - \alpha  + k} \left( {\lambda \sqrt {t^2  - \left| {\xi  - x} \right|^2 } } \right)}}
{{\left[ {t^2  - \left| {\xi  - x} \right|^2 } \right]^{\alpha  - k} }}g_k^* (\xi )d\xi }
\]
где  $g_k^* (s) = \sum\limits_{j = 0}^k {( - 1)^j C_k^j [B_{\alpha  - 1/2}^x ]^j \Psi _{k - j}^* (x)},$
$\Psi _k^* (x) = \sum\limits_{j = 0}^k {a_j C_k^j \lambda ^{2(k - j)} \psi _j^* (x)}, $
$k = \overline {0,m - 1}. $

Аналогично, в случае четного $n,$ из \eqref{sc47} имеем
\[
u_2 (x,t) = \bar \gamma _n \sum\limits_{k = 0}^{m - 1} {\frac{{2^{ - 2k} }}
{{k!((3/2) - \alpha )_k }}}  \times
\]
\[
 \times \left( {\frac{1}{t}\frac{\partial }{{\partial t}}} \right)^{\frac{n}
{2}} \int\limits_{\left| {\xi  - x} \right| < t} {\left[ {t^2  - \left| {\xi  - x} \right|^2 } \right]^{(3/2) - \alpha  + k} \bar J_{(3/2) - \alpha  + k} \left( {\lambda \sqrt {t^2  - \left| {\xi  - x} \right|^2 } } \right)g_k^* (\xi )d\xi }.
\]

В заключении заметим, что данный метод решения поставленной задачи можно применить и в том случае, когда в уравнении \eqref{eq1} оператор Бесселя, действует по нескольким пространственным переменным.

\newpage

\begin{center}
{\bf A.K.Urinov, Sh.T.Karimov}\\
\textbf{Solution of the analogue of the Cauchy problem for the iterated multidimensional Klein-Gordon-Fock equation with the Bessel operator}\\[2mm]

\end{center}

{\bf Abstract.} An analogue of the Cauchy problem for the iterated multidimensional Klein-Gordon-Fock equation with a time-dependent Bessel operator is investigated. Applying the generalized Erd\'{e}lyi-Kober operator of fractional order, the problem posed is reduced to the Cauchy problem for the poly-wave equation. An explicit formula for solving this problem is constructed by the spherical mean method. On the basis of the solution obtained, an integral representation of the solution of the problem is found.

{\it\bf Key Words}: Cauchy problem, Klein-Gordon-Fock equation, generalized Erd\'{e}lyi-Kober operator, Bessel operator.

\end{document}